\documentclass{amsart} 
\def \R{\mathbb{R}}

\def \C{\mathbb{C}}  
\usepackage{soul}
\usepackage{xcolor}

\def\co{\colon}

\usepackage[normalem]{ulem} 

\usepackage{epsf} 
\usepackage{amssymb,latexsym, amsmath, amscd, array, graphicx}

\swapnumbers 
\numberwithin{equation}{section} 

\theoremstyle{plain} 
\newtheorem{thm}{Theorem}[section] 
\newtheorem{cor}[thm]{Corollary} 
\newtheorem{theorem}[thm]{Theorem}

\theoremstyle{definition}

\newtheorem{example}[thm]{Example} 
\newtheorem{conjecture}[thm]{Conjecture}

\newtheorem{remark}[thm]{Remark}

\def\C{{\mathbb C}}

\def\R{{\mathbb R}}

\def\1{\hbox{\rm\rlap {1}\hskip.03in{\rom I}}} 
\def\Bbbone{{\rm1\mathchoice{\kern-0.25em}{\kern-0.25em} 
{\kern-0.2em}{\kern-0.2em}I}} 

\begin{document}
\title[Conjectures about virtual Legendrian knots and links] 
{Conjectures about virtual Legendrian knots and links} 
\author[V.~Chernov, R.~Sadykov]{Vladimir Chernov, Rustam Sadykov}
\address{V.~Chernov, Mathematics Department, 6188 Kemeny,  Dartmouth College, Hanover NH 03755, 
USA} 
\email{Vladimir.Chernov@dartmouth.edu} 

\address{R.~Sadykov, Mathematics Department
138 Cardwell Hall
1228 N. Martin Luther King Jr. Dr.
Manhattan, KS 66506--2602, 
USA} 
\email{Sadykov@math.ksu.edu}

\subjclass{Primary 53C24; Secondary 53C50, 57D15, 83C75}

\begin{abstract} We formulate conjectures generalizing some known results to the category of virtual Legendrian knots. This includes statements relating virtual Legendrian knots to ordinary Legendrian knots, non-existence of positive virtual Legendrian self isotopies for the class of the fiber of $ST^*M$ and the conjectural relation of virtual Legendrian isotopies to causality in generalized spacetimes. We prove the conjectures in the case of $2$-dimensional $M$ and $(2+1)$-dimensional spacetimes. We also formulate and prove the version of the Arnold's $4$ cusp conjecture for virtual isotopies.
\end{abstract}

\keywords{virtual knot, Legendrian knot, causality, spacetime, linking, Low conjecture, positive Legendrian isotopy, non-negative Legendrian isotopy}

\maketitle

\section{Main Results and Definitions}
We work in the smooth category, all maps and manifolds are assumed to be smooth and the word ``smooth'' means $C^{\infty}.$ All manifolds, links and knots are assumed to be oriented and connected (unless we talk about links).

A {\it contact manifold} $N^{2k+1}$ is an odd dimensional manifold equipped with a hyper plane distribution $\xi$ that can locally be given as the kernel of a $1$-form $\alpha$ with $\alpha\wedge d\alpha^k$ nowhere zero. A $k$-dimensional manifold $L^k$ is {\it Legendrian} if it is everywhere tangent to the hyperplanes of the contact structure and it is {\it transverse\/} if it is nowhere tangent to the  contact hyperplanes. 

The standard examples of contact manifolds are the spherical cotangent bundle $ST^*M$ and the first jet space $J^1M$ which have natural contact structures.
Note that a surgery on $M$ resulting in a manifold $M'$ induces a modification of contact manifolds $ST^*M\mapsto ST^*M'$.

{\it A virtual Legendrian isotopy\/} in $ST^*M$ is a sequence of Legendrian isotopies, modifications of Legendrian submanifolds in $ST^*M$ induced by surgery on $M$ away from the front projection of the Legendrian submanifold, and  contactomorphisms of $ST^*M$ induced by diffeomorphisms of the manifold $M$.
Virtual Legendrian knots were introduced in the work of Cahn and Levi~\cite{CahnLevi}. 

The virtual Legendrian isotopy is {\em subcritical} if the index of the attached handles is less than half of the dimension of $M$  and it is {\em allowable} if the index of the attached handles does not exceed half of the dimension of $M.$ In particular, the virtual Legendrian isotopies of~\cite{CahnLevi}  that do not involve deleting handles are allowable. {\em All the conjectures we formulate are for the allowable virtual Legendrian isotopies.\/}

Note that a surgery on $M$ induces a modification on $J^1M$ and virtual Legendrian isotopies in $J^1M$ are defined similarly to the virtual Legendrian isotopies on $ST^*M$. In a similar fashion one can define virtual transverse isotopies both for $ST^*M$ and $J^1M.$
We note that in contrast to virtual isotopies in $ST^*M$, in the case of virtual Legendrian isotopies in $J^1M$ the projection of a virtual Legendrian knot to $M$ is of codimension $0$. In both cases, the surgery of virtual Legendrian knots is performed on the complement in $M$ to the projection of the virtual Legendrian knot. 

Classical virtual knots were introduced by Kauffman~\cite{Kauffman} and they are ordinary knots in a thickened surface $F\times \R$ considered up to a sequence of isotopies and modifications of $F\times \R$ induced by surgery on $F$ away from a projection of a knot (or a link). A Gauss diagram of a knot is an oriented chord diagram on a circle. The chords connect the preimages of double points of a diagram projection and are oriented from the preimage of an upper strand to the preimage of a lower strand. They are equipped with signs $\pm 1$ which is the writhe of the corresponding diagram crossing. Polyak and Viro
~\cite{PolyakViro} noted that a Gauss diagram of a knot uniquely defines the knot on $S^2\times \R.$ 

Gauss diagrams can be considered up to the formal diagramatic version of the usual knot Redemeister moves however using these moves it is easy to get a Gauss diagram which does not correspond to a planar knot diagram and hence is virtual. Carter, Kamada, Saito~\cite{CarterKamadaSaito} proved that formal Gauss diagrams considered up to these formal Redemeister moves give the same equivalence classes as the ones of virtual knots. 

As it is well known the cooriented front projection of a Legendrian knot in $ST^*M$ to $M$ completely determines the Legendrian knot. Such front projections on a surface can be described by diagrams similar to Gauss diagrams with the symbols on chords denoting the crossing and cusp types. Cahn and Levi~\cite{CahnLevi} proved that equivalence classes of such generalized ``Gauss diagrams'' for front projections to a surface coincide with the equivalence classes of virtual Legendrian knots in $ST^*M.$

A natural question is if a similar result holds for virtual Legendrian knots and links in $J^1M$ or for $ST^*M$ for higher dimensional manifold $M$ closed or non-closed.

\section{Virtual Positive Legendrian Isotopy}

Recall that a $Y^x_{\ell}$ manifold is a Riemannian manifold $(M,g)$ such that there is a point $x\in M$ and a positive number $\ell>0$ such that all unit speed geodesics starting from $x$ return back to $x$ in time $\ell.$ For such manifolds the cogeodesic flow in $ST^*M$ induces a positive Legendrian isotopy of the sphere fiber over $x$ to itself. Recall that a Legendrian isotopy is {\it positive} if the evaluation of the contact form on all the velocity vectors of all the points and at all times of the deformed Legendrian submanifold is positive. {\em Non-negative\/} Legendrian isotopies are defined similarly.

The classical result of Berard-Bergery~\cite{BerardBergery},~\cite{Besse} says that if $M$ admits a Riemannian metric making it into a $Y^x_{\ell}$ manifold then the universal cover of $M$ is compact (and hence $M$ has finite $\pi_1(M))$ and the rational cohomology algebra of its universal cover is generated by one element. In the works of the first author and Nemirovski we proved~\cite{CN1, CN2} that if $M$ is such that there is a non-constant non-negative Legendrian isotopy of a fiber of $ST^*M$ over a point of $M$ back to itself then the universal cover of $M$ is compact. So in this work we conjectured that if such a non-negative Legendrian isotopy exists, then the rational cohomology algebra of the universal cover is generated by one element. 

A slightly stronger version of this conjecture for positive Legendrian isotopies was proved by Frauenfelder, Labrouse and Schlenk~\cite{FLS} who formally proved more than we conjectured and showed that the integer cohomology ring of the universal cover is the one of a CROSS (Compact Rank One Symmetric Space), i.e. a sphere or one of the projective spaces over $\mathbb C, \mathbb H, \mathbb Ca.$
In particular this means that the conclusions of the Berard-Bergery theorem and of the Bott-Samelson theorem~\cite{BottSamelson},~\cite{Besse} are the same. The result of~\cite{FLS} can be generalized to non-negative Legendrian isotopies by using the result of Nemirovski and the first author~\cite{CN3} saying that if there is a non-negative Legendrian isotopy then there is a positive Legendrian isotopy as well, even though this positive Legendrian isotopy is obtained not by deforming the non-negative Legendrian isotopy but rather by deforming a certain power of it. (See also Colin, Ferrand and Pushkar~\cite{CFP} and~\cite{CN1} for the case of positive isotopies of a fiber to itself in the case where $M$ is coverable by an open domain in $\R^m.$)

Currently the only simply connected spaces known to us  that satisfy these algebraic topology conditions but are not diffeomorphic to a CROSS are exotic spheres, see Besse~\cite{Besse}. So it could be that this algebraic topology condition means that the simply connected manifold is homeomorphic to a CROSS.

We define {\it strong virtual Legendrian isotopies} to be the virtual Legendrian isotopies during which the universal cover of the underlying manifold $M$ is never homeomorphic to a CROSS.

Now we formulate our first conjecture
\begin{conjecture}\label{c:2.1}  There exists no strong allowable positive (or non-constant non-negative) virtual Legendrian isotopy of the sphere fiber of $ST^*M$ to itself.
\end{conjecture}

\begin{remark}
A similar statement is false if we allow all positive virtual Legendrian isotopies rather than only the allowable ones. To see this consider a positive Legendrian isotopy of the fiber over a point of a surface $M^2$ to the knot whose front projection is a small outwards cooriented circle $K.$ Let $K_-$ and $K_+$ be circular fronts with the same center as $K$ cooriented outwards with radii respectively a bit smaller and a bit larger than $K.$  Now perform an index one surgery such that one of the feet of the glued handle is the disk bounded by $K_-.$ Next perform the index two surgery along $K_+.$ As a result we get a small inwards cooriented front which clearly admits a positive Legendrian isotopy to the sphere fiber. Thus positive virtual Legendrian isotopies of a fiber of $ST^*M$ to itself do exist if we do not require the virtual Legendrian isotopy to be allowable. It could however be that to avoid this phenomenon it is enough to prohibit only some of the surgeries involving handles of index more than half of the dimension of $M^m$ (for example index $m$ handles) rather than all of them.
\end{remark}

\begin{theorem}\label{theorempositiveisotopy}
The statement of Conjecture~\ref{c:2.1} is true in the case where $M$ is a surface. 
\end{theorem}

\begin{proof}
Suppose that $M$ is a closed surface. Assume to the contrary that a spherical fiber $K_0$ of $ST^*M$ admits an allowable non-negative virtual Legendrian isotopy to another fiber. We may assume that the surface $M$ is connected, and consider only surgeries of index $1$ which we call stabilizations. Then an allowable non-negative virtual Legendrian isotopy is a composition of finitely many stabilizations of the surface and positive Legendrian isotopies. 

Let $M_k$ denote the surface obtained from $M$ by $k$ stabilizations. 
Suppose that the number of stabilizations is $n$, and $n$ is minimal in the sense that there is no allowable non-negative virtual Legendrian isotopy with less than $n$ stabilizations of $K_0$ to another fiber. 
Then a part of the allowable non-negative Legendrian isotopy takes $K_0$ to a Legendrian knot $K_{n-1}$ in $ST^*M_{n-1}$ that satisfies two properties: 
\begin{itemize} 
\item the Legendrian knot $K_{n-1}$ does not admit a non-negative Legendrian isotopy to a fiber in $ST^*M_{n-1}$, and 
\item there is a non-negative Legendrian isotopy of the stabilization $K_n\subset ST^*M_n$ of the knot $K_{n-1}$ to a fiber in $ST^*M_n$. 
\end{itemize}

The surface $M_{n}$ is obtained from the surface $M_{n-1}$ by removing two discs $2D$ in the complement to $p(K_{n-1})$, where $p$ is the projection of $ST^*M_{n-1}$ to $M_{n-1}$, and attaching along the created boundary a handle $H=[-1,1]\times S^1$. Let $C$ denote the meridian $\{0\}\times S^1$ on the handle $H$. There is an infinite cover $\pi\co M_n\setminus C\to M_n$ that restricts to the identity map on $M_n\setminus H$. 
Since the knot $K_n$ is isotopic to a fiber in $ST^*M_n$, it admits a lift $\pi^*K_n$ to a Legendrian knot in $\pi^*ST^*M_n$. 
Furthermore, the non-negative Legendrian isotopy of $K_n$ to a fiber in $ST^*M_n$ lifts to a non-negative Legendrian isotopy of $\pi^*K_n$ to a fiber in $\pi^*ST^*M_n$. 

On the other hand, the space $\pi^*(ST^*M_n)$ can be identified with the spherical cotangent bundle of $M_n\setminus C$, which in its turn is contactomorphic to the spherical cotangent bundle of $M_n\setminus H \cong M_{n-1}\setminus 2D$. Furthermore, we may choose a contactomorphism $\psi\co \pi^*(ST^*M_n)\cong ST^*(M_{n-1}\setminus 2D)$ so that it covers a diffeomorphism $M_n\setminus C\to M_{n-1}\setminus 2D$ that restricts to the map that identifies the complement to a small neighborhood of $H$ in $M_{n}$ with the complement to a small neighborhood of $2D$ in $M_{n-1}$. In particular, the contactomorphism $\psi$ takes the Legendrian knot $K_n$ to the Legendrian knot $K_{n-1}$. 
Then the non-negative Legendrian isotopy of $K_n$ to a fiber in $ST^*M_n$ lifts to a non-negative Legendrian  isotopy of $K_{n-1}$ to a fiber in $ST^*M_{n-1}$.  

This contradicts the minimality of $n$. Therefore the statement of Conjecture~\ref{c:2.1} is true for closed surfaces. 

Suppose that the surface $M$ is noncompact of finite genus.  Then it is an open submanifold of a closed surface $M'$ of finite genus. As above, we obtain surfaces $M'_{n-1}$ and $M'_{n}$ where $M'_{n}$ is obtained from $M'_{n-1}$ by a surgery of index $1$. Again, let $C$ be a meridian of the attached handle.  Then there is a cover $\pi'\co M_n'\setminus C\to M_n'$. 
Similarly, we obtain open submanifolds $M_{n-1}\subset M'_{n-1}$ and $M_n\subset M'_n$. We may assume that $C\subset M_n$. Then there is a cover $\pi\co M''\to M_n$ obtained by restricting the map $\pi'$ to $(\pi')^{-1}M_n$. On the other hand, the surface $M''$ can be identified with a subsurface of $M\setminus 2D$, where $2D$ is a disjoint union of two discs along which the $1$-handle is attached. Therefore, the non-negative Legendrian isotopy of a knot in $ST^*M_n$ lifts to a non-negative Legendrian isotopy of the corresponding knot in $ST^*M_{n-1}$. 

Suppose now that $M$ is of infinite genus. Then the trace of a non-negative virtual Legendrian isotopy is contained in a compact part of all the corresponding surfaces. Thus this case can be reduced to the case where $M$ is of finite genus. 
\end{proof}

\begin{cor}\label{corollarypositive}
Let $M$ be a surface and let $L$ be a two component Legendrian link in $ST^*M$ that admits a strong allowable virtual Legendrian isotopy to a pair of sphere fibers. Then there is {\bf no} non-negative strong allowable virtual Legendrian isotopy of one component of $L$ to the other.
\end{cor}
\begin{proof}
The proof is similar to the one of Theorem~\ref{theorempositiveisotopy}. Indeed, the argument in the proof of Theorem~\ref{theorempositiveisotopy} shows that the existence of an allowable virtual Legendrian isotopy of $L$ to a pair of sphere fibers implies the existence of a Legendrian isotopy of $L$ to a pair of fibers in $ST^*M$. Consequently, there is a contactomorphism $c_1$ of $ST^*M$ that takes $L$ to a link $L'$ that consists of two sphere fibers of $ST^*M$.

Similarly, the argument in the proof of Theorem~\ref{theorempositiveisotopy} shows that the existence of a non-negative allowable virtual Legendrian isotopy from one component of $L$ to the other implies the existence of a non-negative Legendrian isotopy $\varphi_t$ from one component of $L$ to the other in $ST^*M$. Then $c_1\circ \varphi_t$ is a non-negative Legendrian isotopy from one sphere fiber in $L'$ to the other, say a fiber $F$ to the fiber $F'$. 

There is a contactomorphism $c_2$ from $ST^*M$ to itself that takes the fiber $F$ to $F'$ and takes $F'$ to $F$. Then $c_2\circ c_1\circ \varphi_t$ is a non-negative Legendrian isotopy that takes the fiber $F'$ to $F$. 

Finally, the composition of non-negative Legendrian isotopies $c_1\circ \varphi_t$ and $c_2\circ c_1\circ \varphi_t$ is a non-constant non-negative Legendrian isotopy from a fiber of $ST^*M$ to itself.  Consequently,  $M$ is a sphere, see~\cite{CN1}. However Corollary~\ref{corollarypositive} is about {\bf strong} allowable virtual Legendrian isotopy so $M$ can not  be a sphere. Thus, there is no non-negative strong allowable virtual Legendrian isotopy of one component to the other.
\end{proof}

\begin{conjecture}\label{nononnegativeisotopy} We conjecture that for $M^m, m\geq 3$, whose universal cover is not homeomorphic to CROSS and  $L$ that is a two component Legendrian link in $ST^*M$ that admits an allowable virtual Legendrian isotopy to a pair of fibers, there is {\bf no} non-negative {\bf strong} allowable virtual Legendrian isotopy of one component of $L$ to the other.
\end{conjecture}

\section{Generalization of Kuperberg result to virtual Legendrian knots and links in higher dimensions}

 By Kuperberg theorem~\cite{Kuperberg}, if two ordinary virtual knots in $M^2\times \R$ are virtually isotopic then they both can be simplified by a sequence of isotopies and modifications induced by surgeries on $M^2$ that decrease  to the situation where two knots are related by an automorphism of $M^2\times \R$ induced by an automorphism of $M^2.$ The similar result for virtual Legendrian knots { \bf in $ST^*M^2$, where $M^2$ is closed} was proved by us in~\cite{ChernovSadykovVirtual}. 

In particular we got the following corollary: if two Legendrian knots in $ST^*S^2$ are virtual  Legendrian isotopic then they are Legendrian isotopic. This follows immediately since the genus of $S^2$ can not be further decreased. 

For higher dimensional virtual Legendrian knots it is not clear what surgeries decrease the complexity of the underlying manifold or even what this complexity is. However below we formulate the conjecture generalizing the above corollary in our work.

\begin{conjecture} Two Legendrian links in $ST^*S^m$ or $ST^*\C P^n$ are Legendrian isotopic if and only if they are virtually Legendrian isotopic. (This conjecture of the first author in the case of spheres was first formulated in~\cite{CahnLevi}.) The same conjecture is formulated for Legendrian links in $J^1S^m$ and $J^1\C P^n.$ It is also interesting to know if a similar statement holds for the fibers of $ST^*M$ where $M$ is a CROSS.
\end{conjecture}

\section{Virtual Legendrian isotopies and causality}

Low~\cite{Low0, Low1, Low2, Low3} formulated the conjecture relating topological linking of the skies to causality. Recall that a \emph{sky} is the sphere of light rays through a point, and it is located in the contact manifold of all future directed unparameterized  light rays (null geodesics). 
Recall that two points in a spacetime are causally related if  information traveling with the light or less can reach from one event point to the other, i.e. if there is a causal (non-spacelike) curve between these points. 

Low's conjecture is for globally hyperbolic spacetimes which are defined as spacetimes without closed causal trajectories (time travel machines) and with the assumption that the intersection of causal future and causal past of any two points is compact (absence of naked singularities when a traveling particle disappears  from the observable universe of the observer without crossing their line of sight), see~\cite{BS3}, or \cite{HE} for a more traditional and equivalent definition. Globally hyperbolic spacetimes form the most important and studied class of spacetimes and one of the versions of Penrose's Strong Cosmic Censorship conjecture~\cite{Penrose} says that all physically relevant spacetimes are globally hyperbolic. (The part of a spacetime located inside of black holes is not relevant to us since the particles can not escape the black hole horizons.) 

Geroch~\cite{Geroch} proved that globally hyperbolic spacetimes are homeomorphic to $\Sigma \times \R$ and Bernal-Sanchez~\cite{BS1, BS2} proved that it is in fact diffeomorphic to $\Sigma \times \R$ with each $\Sigma \times t$ being a {\em spacelike Cauchy surface\/} i.e. a codimension one submanifold intersected at exactly one point by every unextendible causal curve and such that the restriction of the Lorentz metric to it is Riemannian. 

As it was observed by Low for such spacetimes the contact manifold of all future directed unparameterized light rays (light geodesics) is contactomorphic to $ST^*\Sigma$ for every spacelike Cauchy surface $\Sigma.$

Low conjecture is for topological linking 
in $(2+1)$-dimensional spacetimes and $\Sigma$ being an open $2$-dimensional surface. This conjecture was reformulated for Legendrian linking for $(3+1)$-dimensional spacetimes with $\Sigma$ homeomorphic to $\R^3$ by Natario and Tod~\cite{NatarioTod}. All these conjectures were proved in the works of first author and Nemirovski~\cite{CN1, CN2, Chernov} and the current state of knowledge is that Legendrian linking is equivalent to causality for globally hyperbolic spacetimes of all dimensions provided that the universal cover of the Cauchy surface is not compact {\bf or} that it is compact but the integer cohomology of the universal covering is not the one of a CROSS.

Antoniou, Kauffman and Lambropolou~\cite{AKL} suggested that singular points of a Morse function on a spacetime can be considered as singularities corresponding to formation of black holes or something similar. The classical spacetimes are time oriented Lorentz manifolds. We generalize this notion as follows, a {\it generalized spacetime} is a manifold equipped with a bilinear symmetric form.   Furthermore we assume that it can be equipped with a Morse function such that away from the critical points the form restricts to a Lorentz metric, we also assume that the Morse function is a timelike function away from the critical points, i.e. the velocity vector of the flow lines of the gradient of the Morse function dot product with itself is negative. We also assume that the non-critical level sets of the Morse function are spacelike Cauchy surfaces. This definition is in the spirit of Borde et al~\cite{Borde} and Sorkin~\cite{Sorkin} Morse spacetimes. (See also Garcia-Haveling~\cite{GH} for some recent developments on the subject of Morse spacetimes.)

We do not quite agree with~\cite{AKL} on the interpretation of surgeries on a spacelike manifold as black hole formations, (because the topology of the level set of a Morse function does not have to change under passage through a black hole formation). However we observe that if the chunks of a generalized spacetime between critical points of a Morse function $f$ are globally hyperbolic and the indexes of the critical points of the Morse function  $-f$ are such that they correspond to allowable surgeries on level sets; then a sphere of light rays through a point considered as a Legendrian knot in $ST^*M$ (for $M$ that is the spacelike level set of a noncritical level of the Morse timelike function) experiences an allowable virtual Legendrian isotopy when passing through the critical level in the direction from higher levels to lower levels. We call such generalized spacetimes {\em allowable}.

Clearly the definition of causality still makes sense for such generalized spacetimes and we formulate

\begin{conjecture}\label{conjecturecausality}
Two points in a generalized allowable spacetime are causally related if and only if the Legendrian link of spheres of light rays through these two points  is not strongly allowable virtually Legendrian isotopic to the link of two fibers of $ST^*M_t$ of the level set manifold $M_t$. These Legendrian spheres can be considered in $ST^*M_t$ for any $M_t$ that is a non-critical level of the Morse function. We call such links non-trivial. 

The last links correspond to causally unrelated events. (The reason why we have to consider strong allowable virtual Legendrian isotopy rather than usual virtual Legendrian isotopy is that  Legendrian linking is not equivalent to causality when a Cauchy surface admits a $Y^x_{\ell}$ Riemann metric~\cite{CN1}.)

Note also that if the two points are on the common light ray then the corresponding link of spheres of light rays has a double point and hence is singular. We call such singular links nontrivial as well the links described above.
\end{conjecture}

\begin{remark}[How can one try to prove the above conjecture]\label{howtoprovecausality}
This conjecture is closely related to the one about nonexistence of allowable strong non-negative non-constant virtual Legendrian isotopy of a fiber to itself. The relation is similar to the ones of the paper of Nemirovski and the first author~\cite{CN1}. Namely the transport of a point in a spacetime to the past with the speed less or equal than light (causal trajectory) induces a non-negative Legendrian isotopy of the sphere of all light rays through the point in the contact manifold of all light rays, see~\cite{CN1} for timelike trajectories and~\cite{CN4} for causal trajectories. 

Given a Morse function $f$ on a manifold $M$ and any subset $I\subset \R$, let $M_{I}$ denote the subset $f^{-1}(I)$ of $M$. To simplify notation, we will write $M_{t}$ instead of $M_{\{t\}}$ in the case where $I=\{t\}$.

Let $p$ and $q$ be two events in a generalized spacetime $M$ equipped with a Morse function $f$ that is timelike in the complement to critical points. Suppose that $f(p)<f(q)$. We may slightly perturb the function $f$ so that both events $p$ and $q$ belong to regular levels of the function $f$. 
Let $c_1<\cdots <c_{n}$ be all critical values in the interval $[f(p), f(q)]$. Choose regular values $a_i$ so that 
\[
   f(p)=a_0<c_1<a_1<c_2<\cdots a_{n-1}< c_{n}<a_n=f(q).
\]
We observe that the manifold  $M_{(c_i, c_{i+1})}$ is globally hyperbolic, while the manifold of all light rays $\mathbf{L}_i$ on $M_{(c_i, c_{i+1})}$ can be canonically parametrized by the spherical cotangent bundle $ST^*M_{a}$ of the level $M_{a}$ for any $a\in (c_i, c_{i+1})$. These parametrizations are quite different. For example, a fiber of the cotangent bundle with respect to one canonical parametrization is not a fiber of the cotangent bundle with respect to another canonical parametrization.
For a point $x$ in $M_{[a_0, a_n]}$ we will denote the sphere of light rays in $\mathbf{L}_i$ emanated from $x$ by $S_i(x)$. 

Suppose that the events $p$ and $q$ are causally unrelated. Choose a future directed timelike path $q^t$, where $t\in [a_0, a_n]$, such that $f(q^t)=t$, $q^{a_n}=q$, and the path $q^t$ does not pass through critical points of the function $f$. As $t$ increases in the interval $[a_0, c_{1})$, the sphere $S_0(q^t)$ is modified by a Legendrian isotopy in $\mathbf{L}_0$. Similarly, as $t$ increases in $(c_i, c_{i+1})$ for $i=1,..., n-1$, the path $q^t$ defines a Legendrian isotopy of $S_i(q^t)$ in $\mathbf{L}_i$, and as $t$ increases in $(c_n, a_n]$ the path $q^t$ defines a Legendrian isotopy of $S_n(q^t)$ in $\mathbf{L}_n$. Also as $t$ varies in the considered intervals the sphere of light rays $S_i(p)$ in $\mathbf{L}_i$ for $i=0,..., n$ remains constant. 

Recall that the manifold of light rays $\mathbf{L}_i$ is canonically identified with the manifold $ST^*M_{a_i}$. In particular, every self-diffeomorphism of $M_{a_i}$ induces a contactomorphism of $ST^*M_{a_i}$.  Recall also that an elementary virtual Legendrian modification of a knot $K$ in $ST^*N$ consists of a modification of $ST^*N$ induced by a surgery on $N$ away from the front projection of $K$ resulting in a manifold $N'$, together with a contactomorphism of $ST^*N'$ induced by a self-diffeomorphism of $N'$. 

It follows that up to a contactomorphism of $\mathbf{L}_{i+1}$ induced by a diffeomorphism of the level $M_{a_{i+1}}$, the pair $(\mathbf{L}_{i+1}, S_{i+1}(q^{a_{i+1}}))$ is obtained from  $(\mathbf{L}_i, S(q^{a_{i}}))$ by elementary virtual Legendrian modifications. We emphasize that at this step we do not identify light rays in $ST^*M_{a_{i+1}}$ with light rays in 
$ST^*M_{a_{i}}$. Instead, we make a discrete transformation using a handle decomposition of $M_{[a_i, a_{i+1}]}$ associated with the restricted Morse function $f|M_{[a_i, a_{i+1}]}$. Each handle corresponds to one elementary virtual Legendrian modification. 
Under such a transformation, the complement to a neighborhood of the attaching sphere in $M_{a_{i}}$ is identified with the complement to a neighborhood of the belt sphere in $M_{a_{i+1}}$. This identification extends to an identification  of spherical cotangent bundles. Under this identification the fiber $S_{i+1}(q^{a_{i+1}})$ of the cotangent bundle $\mathbf{L}_{i+1}=ST^*M_{a_{i+1}}$ is identified with the fiber $S(q^{a_i})$ of the cotangent bundle $\mathbf{L}_{i}=ST^*M_{a_{i}}$. 

Similarly, we claim that as $t$ increases from $c_{i+1}-\varepsilon$ to $c_{i+1}+\varepsilon$ crossing a critical level $c_i$ for $i=1,..., n$, the sphere of light rays $S_i(p)$ changes to a sphere of light rays $S_{i+1}(p)$ by finitely many elementary virtual Legendrian modifications. Indeed, we may assume that the critical level $c_i$ contains only one critical point. Then the complement to a neighborhood of the attaching sphere in $M_{c_{i+1}-\varepsilon}$ can be identified with the complement to a neighborhood of the belt sphere in $M_{c_{i+1}+\varepsilon}$ by means of the gradient flow of $f$.  The above identification extends to an identification $\psi$ of spherical cotangent bundles. Under this identification the sphere of light rays $S_{i+1}(p)$ in $\mathbf{L}_{i+1}=ST^*M_{c_{i+1}+\varepsilon}$ may not coincide with the image $\psi(S_i(p))$ of the sphere $S_i(p)$ in  $\mathbf{L}_{i}=ST^*M_{c_{i+1}-\varepsilon}$ since the trajectories of the gradient flow of $f$ in the definition of $\psi$ are fairly arbitrary timelike trajectories that do not have much relation to the cones of light rays.
 However, this can be rectified by a contactomorphism of $\mathbf{L}_{i+1}$ induced by a self-diffeomorphism of $M_{c_{i+1}+\varepsilon}$.

Thus, we have constructed a virtual Legendrian isotopy of the sphere $S_0(p)$ of light rays to the sphere $S_n(p)$ and a virtual Legendrian isotopy of the sphere $S_0(q^0)$ to $S_n(q)$. 
Furthermore, for each $i$, the sphere of light rays of $p$ in $\mathbf{L}_i$ is disjoint from the sphere of light rays of $q$ in $\mathbf{L}_i$ since otherwise there would exist a causal trajectory from $p$ to $q$ for causally unrelated events $p$ and $q$. Thus, we have constructed a virtual Legendrian isotopy of the link of two fibers $(S_0(p), S_0(q^0))$ to the link $(S_n(p), S_n(q))$.   
The reverse isotopy is a strong allowable virtual Legendrian isotopy to a link of two fibers. 

To prove the converse statement of the conjecture, suppose that $p$ and $q$ are causally related. Then there is a causal path $q^t$ with $q^0=p$ and $q^n=q$. The argument similar to the one above shows now that there is a non-trivial non-negative Legendrian isotopy of the fiber $S_n(q^n)$ to the fiber $S_0(q^0)$. Indeed, as $t$ varies in the past direction in one of the intervals $[a_0, c_1), (c_i, c_{i+1})$ or $(c_n, a_n]$, the sphere $S_{i}(q^t)$ is modified by a non-trivial non-negative virtual Legendrian isotopy. As $t$ passes a critical point $c_{i+1}$ in the negative direction, the sphere $S_{i+1}(q^t)$ is modified by finitely many elementary virtual Legendrian modifications to a sphere $S_{i}(q^t)$. We emphasize that we never consider the sphere $S_n(q^t)$ for small values of $t$ since the light rays of $q^t$ may not reach the level $a_n$ because of the critical points. As $t$ becomes smaller than $c_1$, the sphere $S_0(q^t)$ turns into the fiber $S_0(q^0)$ in $\mathbf{L}_0=ST^*M_{a_0}$. 

Next we postcompose the resulting non-trivial non-negative virtual Legendrian isotopy from $S_n(q^n)$ to $S_0(q^0)$ with the virtual Legendrian isotopy of $S_0(p)=S_0(q^0)$ to $S_n(p)$ that consists of elementary virtual Legendrian modifications and contactomorphisms of cotangent bundles induced by diffeomorphisms of underlying manifolds. The resulting isotopy $\varphi$ is a non-trivial non-negative virtual Legendrian isotopy of one component of the link $(S_n(p), S_n(q))$ in $\mathbf{L}_n$  to the other. 

Assume now contrary to the statement of the conjecture that despite the existence of a causal path from $p$ to $q$ the link $(S_n(p), S_n(q))$ is 
an unlink, i.e., there is a contactomorphism $\psi$ of $\mathbf{L}_n$ that takes the link $(S_n(p), S_n(q))$ to the link that consists of two fibers. Then $\psi$ takes $\varphi$ to a non-negative non-trivial Legendrian isotopy that takes one fiber to another. 

By Corollary~\ref{corollarypositive}, this implies the conjecture for $2+1$-dimensional allowable generalized spacetimes.

For higher dimensional allowable generalized spacetimes the conjecture follows in a similar way provided that Conjecture~\ref{nononnegativeisotopy} is true.

\end{remark}

In particular we proved the following Theorem
\begin{theorem}\label{causalitytheorem}
The statement of conjecture~\ref{conjecturecausality} is true for $(2+1)$-dimensional generalized allowable spacetimes.
\end{theorem}

\begin{remark}
The proposed approach to the proof of Conjecture~\ref{conjecturecausality} is primarily based on the conjecture that there is no non-negative allowable virtual Legendrian isotopy of a sphere fiber of $ST^*M$ to itself. Allowable virtual isotopies prohibit surgeries on $M$ that involve gluing handles of index more than half of the dimension of $M.$ However the only example we know where in the usual non-negative non-constant virtual Legendrian of a fiber to itself exists involves handles of index equal to the dimension of $M.$ So it could be that these conjectures are true if one relaxes the definition of the allowable virtual Legendrian isotopy and prohibits only some of the surgeries involving handles of index more than half of the dimension of $M^m$ (for example index $m$ handles) rather than all of them.
\end{remark}

\section{A few words about non-existence of the generalization of Arnold's $4$-cusp conjecture}

Arnold~\cite{Arnold} formulated the following conjecture. Consider two Legendrian knots in $ST^*\R^2$ that are the slightly perturbed circle-fibers that project to the inward and the outward cooriented oval fronts. Then every generic Legendrian isotopy between them involves a moment where the front projection has four cusps.

Chekanov and Pushkar~\cite{ChekanovPushkar} proved this conjecture.

Using standard covering techniques $\R^2\to F^2$ it is clear that Arnold's conjecture is true for all surfaces other than a sphere. For the sphere the cogeodesic flow takes the expanding circle around the North pole to the contracting circle around the South pole. This propagation does not involve any cusps and the statement of the conjecture would be false.

The conjecture also does not hold for virtual Legendrian knots.

Indeed, consider an outward growing small circular planar wave front and perform an index $1$-surgery on the plane by gluing a handle one of whose feet is a disk inside of our front. Propagate the front a bit along the handle and do the reverse surgery via cutting along a circle located above the front. Now our front appears to be a contracting circle and no cusps appeared in this process at all. 


However the following Theorem is true.

\begin{theorem} Let $K_1$ and $K_2$ be two Legendrian knots in $ST^*M^2_1$ and $ST^*M_2$ that correspond to the small circular  inside and outside  cooriented fronts respectively. Assume moreover that $M_i\neq S^2$. Then every generic strong allowable virtual Legendrian isotopy from $K_1$ to $K_2$ involves a moment when the front projection of the knot has at least four cusps.
\end{theorem}

\begin{proof}
Take a strong allowable virtual Legendrian isotopy of $K_1$ to $K_2$. It consists of 
\begin{itemize}
\item elementary Legendrian modifications resulting in surfaces $M_0, M_1, ..., M_n$, 
\item contactomorphisms, which we omit to simplify notations, and 
\item isotopies $\varphi_t^{(i)}$, $t\in [0,1]$ of Legendrian knots $K^{(i)}\subset ST^*M_i$ satisfying $K^{(0)}=K_1$,  $\varphi^{(n)}_1(K^{(n)})=K_2$, and the property that $K^{(i+1)}$ is the stabilization of $\varphi^{(i)}(K^{(i)})$.
\end{itemize}
 For $i=1,..., n$, the manifold $M_{i}$ is obtained from $M_{i-1}$ by attaching a handle.  Let $S_i\subset M_{i}$ denote the meridian of the handle attached to $M_{i-1}$.

As in the proof of Theorem~\ref{theorempositiveisotopy}, we observe that there is a covering $M_{i}\setminus S_i\to M_i$, and that $M_{i}\setminus S_i$ can be identified with a subset of $M_{i-1}$. Therefore an isotopy in $M_i$ of a knot in $M_i\setminus S_i$ can be lifted to an isotopy in $M_{i-1}$. Thus, starting with $i=n$, by induction with decreasing $i$, we may lift the family of isotopies $\varphi_t^{i}$ to a Legendrian isotopy from $K_1$ to a knot with an inside cooriented circular front. By the Arnold's conjecture for surfaces, we deduce that for the resulting Legendrian isotopy there is a moment at which the front projection of the knot has at least four cusps. Therefore, the same is true for the virtual allowable Legendrian isotopy.
\end{proof}

{\bf Acknowledgments.} The authors are grateful to Roman Golovko for the many motivating questions and they are very thankful to the anonymous referee for many corrections and suggestions. 

This work was partially supported by a grant from the Simons Foundation  (\#513272 to Vladimir Chernov).

\end{document}